\documentclass[11pt]{amsart}%
\usepackage{hyperref}
\usepackage{amsmath}%
\usepackage{amsfonts}%
\usepackage{amssymb}%
\usepackage{amsthm}
\usepackage{todonotes}

\renewcommand{\bar}{\overline}
\renewcommand{\tilde}{\widetilde}

\newcommand{\mf}{\mathfrak}
\newcommand{\e}{\varepsilon}

\DeclareMathOperator{\Des}{Des}
\DeclareMathOperator{\des}{des}
\DeclareMathOperator{\Peak}{Peak}

\newcommand{\Z}{\mathbb{Z}}
\newcommand{\R}{\mathbb{R}}
\newcommand{\qnum}[1]{[#1]_q}
\newcommand{\qfac}[1]{[#1]!_q}
\newcommand{\qbinom}[2]{\genfrac{[}{]}{0pt}{}{#1}{#2}_q}

\newtheorem{theorem}{Theorem}
\newtheorem{prop}[theorem]{Proposition}
\newtheorem{conj}[theorem]{Conjecture}
\newtheorem{lemma}[theorem]{Lemma}
\newtheorem{cor}[theorem]{Corollary}

\theoremstyle{definition}
\newtheorem{ex}[theorem]{Example}

\theoremstyle{remark}
\newtheorem*{remark}{Remark}

\begin{document}

\title{On $q$-analogs of descent and peak polynomials}
\author{Christian Gaetz}
\author{Yibo Gao}
\thanks{C.G. was partially supported by an NSF Graduate Research Fellowship under grant No. 1122374}
\address{Department of Mathematics, Massachusetts Institute of Technology, Cambridge, MA.}
\email{\href{mailto:gaetz@mit.edu}{gaetz@mit.edu}} 
\email{\href{mailto:gaoyibo@mit.edu}{gaoyibo@mit.edu}}
\date{\today}

\begin{abstract}
\emph{Descent polynomials} and \emph{peak polynomials}, which enumerate permutations $\pi \in \mf{S}_n$ with given descent and peak sets respectively, have recently received considerable attention \cite{Billey-Burdzy-Sagan, descent-polynomials}.  We give several formulas for $q$-analogs of these polynomials which refine the enumeration by the length of $\pi$.  In the case of $q$-descent polynomials we prove that the coefficients in one basis are \emph{strongly $q$-log concave}, and conjecture this property in another basis.  For peaks, we prove that the $q$-peak polynomial is palindromic in $q$, resolving a conjecture of Diaz-Lopez, Harris, and Insko.
\end{abstract}

\maketitle

\section{Introduction} \label{sec:intro}

For $\pi=\pi_1 \ldots \pi_n$ a permutation in the symmetric group $\mf{S}_n$ written in one-line notation, the \emph{descent set} of $\pi$ is
\[
\Des(\pi)=\{i \in [n-1] \: | \: \pi_i > \pi_{i+1}\},  
\]
where $[n-1]$ denotes the set $\{1,\ldots,n-1\}$; we write $\des(\pi)$ for the number $|\Des(\pi)|$ of descents in $\pi$.  The \emph{length} of $\pi$ is the number of \emph{inversions}:
\[
\ell(\pi)=|\{(i,j) \: | \: 1\leq i < j \leq n, \: \pi_i>\pi_j\}|.
\]
Both statistics $\ell(\pi)$ and $\des(\pi)$ are of fundamental importance in the combinatorics of the symmetric group, as are their generalizations in other Coxeter groups.

The generating polynomial of the statistic $\des$ on $\mf{S}_n$ is known as the \emph{Eulerian polynomial}:
\[
A_n(t)=\sum_{\pi \in \mf{S}_n} t^{\des(\pi)}.
\]
These polynomials can be succinctly encoded in the generating function
\begin{equation} \label{eq:eulerian-gf}
   \sum_{n \geq 0} \frac{x^n}{n!} A_n(t) = \frac{(1-t)e^{x(1-t)}}{1-te^{x(1-t)}}. 
\end{equation}
Keeping track of the joint distribution of $\des$ and $\ell$ gives the following elegant $q$-analog of (\ref{eq:eulerian-gf}) due to Stanley \cite{Stanley-binomial-posets} (see Reiner \cite{Reiner-descents-and-length} for the generalization to Coxeter groups):
\begin{equation} \label{eq:q-eulerian-gf}
    \sum_{n \geq 0} \frac{x^n}{\qfac{n}} \sum_{\pi \in \mf{S}_n} t^{\des(\pi)}q^{\ell(\pi)} = \frac{(1-t)\exp(x(1-t);q)}{1-t\exp(x(1-t);q)}.
\end{equation}
Here $\qfac{n}=\qnum{1} \cdots \qnum{n}$ where $\qnum{k}=1+q+\cdots +q^{k-1}$ and $\exp(x;q)=\sum_{n \geq 0} x^n/\qfac{n}$.

Rather than considering the distribution of $\des$ on $\mf{S}_n$ for each $n$ as in (\ref{eq:eulerian-gf}), MacMahon \cite{MacMahon} showed that if one fixes a finite subset $S \subset \Z_{>0}$ then the function
\[
D_S(n)=|\{\pi \in \mf{S}_n \: | \: \Des(\pi)=S\}|
\]
is a polynomial in $n$.  These \emph{descent polynomials}, although defined in 1915, have received significant attention only recently, beginning with the work of Diaz-Lopez et al. \cite{descent-polynomials} and continuing with a flurry of work (see \cite{Bencs, Jiradilok-McConville, Oguz}) on several open problems raised there.

The \emph{peak set} of $\pi$ is
\begin{align*}
\Peak(\pi)&=\{i \in \{2,3,...,n-1\} \: | \: \pi_{i-1} < \pi_i > \pi_{i+1}\}\\ &=\{i \in \{2,3,...,n-1\} \: | \: i \in \Des(\pi), i-1 \not \in \Des(\pi)\}.
\end{align*}
Similarly to the case of descent sets, we let
\[
P_S(n)=|\{\pi \in \mf{S}_n \: | \: \Peak(\pi)=S\}|.
\]
Billey, Burdzy, and Sagan \cite{Billey-Burdzy-Sagan} proved the remarkable result that, after accounting for a power of two, this function is also polynomial:
\begin{equation}\label{eq:peak-is-power-times-poly}
P_S(n) = 2^{n-|S|-1} p_S(n),
\end{equation}
where $p_S(n)$ is an integer-valued polynomial called the \emph{peak polynomial}.

In light of the elegant $q$-analog (\ref{eq:q-eulerian-gf}) of (\ref{eq:eulerian-gf}), and motivated by a conjecture of Diaz-Lopez, Harris, and Insko (see Corollary \ref{cor:symmetry-conjecture}), this papers studies the natural $q$-analogs of the descent and peak polynomials
\[
    D_S(n,q) = \sum_{\substack{\pi \in \mf{S}_n \\ \Des(\pi)=S}} q^{\ell(\pi)}
\]
and 
\[
    P_S(n,q) = \sum_{\substack{\pi \in \mf{S}_n \\ \Peak(\pi)=S}} q^{\ell(\pi)}
\]
with the aim of understanding them uniformly in $n$ and $q$.  Note that specializing $q=1$ recovers the usual functions $D_S(n)$ and $P_S(n)$.

The remainder of Section \ref{sec:intro} covers background which will be needed later; none of this is new except for Proposition \ref{prop:closed-under-convolution}.  Section \ref{sec:descent-and-peak-background} covers background material on descent and peak polynomials, while Section \ref{sec:strong-q-log-concavity} recalls the notion of the \emph{strong $q$-log concavity} of a sequence of polynomials which is the subject of Theorem \ref{thm:log-concavity-in-b-basis} and Conjecture \ref{conj:log-concavity-in-a-basis}.  

Section \ref{sec:q-descent} provides two formulas, in Theorems \ref{thm:q-des-formula} and \ref{thm:log-concavity-in-b-basis}, for the $q$-analog $D_S(n,q)$ of the descent polynomial in two different bases of $q$-binomial coefficients.  Theorem \ref{thm:log-concavity-in-b-basis} also establishes the strong $q$-log concavity of the coefficients in one of these bases, generalizing a result of Bencs \cite{Bencs}.  This property for the coefficients in the other basis is the content of Conjecture \ref{conj:log-concavity-in-a-basis}.

In Section \ref{sec:q-peak}, Theorem \ref{thm:q-peak-formula} provides a formula for the $q$-analog $P_S(n,q)$ as a weighted count of certain \emph{$S$-compatible sets}, which is new even in the specialization $q=1$.  Corollary \ref{cor:symmetry-conjecture} then resolves a conjecture of Diaz-Lopez, Harris, and Insko by proving that this function is a palindromic polynomial in $q$ for fixed $n$.  Section \ref{sec:alternate-proof} provides an alternate proof of this corollary, with another expression for $P_S(n,q)$ in terms of $q$-Eulerian polynomials given in Lemma \ref{lem:PIE} and Proposition \ref{prop:expression-Q}.

\subsection{Descent and peak polynomials}
\label{sec:descent-and-peak-background}

Let $A(S;n)=\{ \pi \in \mf{S}_n \: | \: \Des(\pi)=S \}$.  In Theorem \ref{thm:q-des-formula} we prove a $q$-analog of the following theorem.

\begin{theorem}[Diaz-Lopez et al. \cite{descent-polynomials}] \label{thm:des-poly-a-basis}
Let $S \subset \Z_{>0}$ be a finite set of positive integers and $m=\max(S)$, then:
\[
D_S(n)=\sum_{k=0}^m a_k(S) {n-m \choose k},
\]
where $a_0(S)=0$ and for $k \geq 1$ the constant $a_k(S)$ is the number of $\pi \in A(S;2m)$ such that $\{\pi_1,...,\pi_m\} \cap [m+1,2m]=[m+1,m+k]$.
\end{theorem}

Conjecture \ref{conj:log-concavity-in-a-basis} suggests a $q$-analog of the following result of Bencs (see Section \ref{sec:strong-q-log-concavity} for the notions of \emph{log-concavity} and \emph{strong $q$-log concavity}).

\begin{theorem}[Bencs \cite{Bencs}]
The sequence $(a_i(S))_{i=0,1,...,m}$ is log-concave.
\end{theorem}

\subsection{Strong $q$-log concavity}
\label{sec:strong-q-log-concavity}

A sequence $(a_i)_{i=0,1,...,k}$ of nonnegative real numbers is \emph{log-concave} if it has no internal zeroes (meaning that if $a_i, a_k \neq 0$ and $i < j < k$, then $a_j \neq 0$) and if $a_i^2 \geq a_{i-1}a_{i+1}$ for all $i$.  This notion for combinatorially-defined sequences is extremely well-studied (see, for example, Stanley's survey \cite{Stanley-log-concave-survey}).

Following Sagan \cite{Sagan-inductive}, we say that a sequence $(a_i(q))_{i=0,...,k}$ of polynomials from $\R_{\geq 0}[q]$ is \emph{strongly $q$-log concave} if it has no internal zeroes and if 
\begin{equation} \label{eq:strong-lc-condition}
    a_i(q)a_j(q)-a_{i-1}(q)a_{j+1}(q) \in \R_{\geq 0}[q]
\end{equation} 
for all $i<j$.  Setting $q$ to 1 gives a log-concave sequence of real numbers, and no generality is lost in this specialization by imposing the condition (\ref{eq:strong-lc-condition}) for all $i<j$, rather than just $i=j$ as in the definition for sequences of real numbers, as the two conditions are equivalent for sequences of real numbers.  It is not true, however, that the case $i=j$ implies the general case for polynomials, as Example \ref{ex:strong-q-log-concave} demonstrates.

\begin{ex} \label{ex:strong-q-log-concave}
Consider the sequence 
\[
(a_0(q),a_1(q),a_2(q),a_3(q))=(2q, 1+q+q^2, 1+q+q^2,2q).
\]
We see that $a_1(q)^2-a_0(q)a_2(q)$ and $a_2(q)^2-a_1(q)a_3(q)$ both have nonnegative coefficients, but $a_1(q)a_2(q)-a_0(q)a_3(q)=1+2q-q^2+2q^3+q^4$ does not.  Thus the case $i=j$ in (\ref{eq:strong-lc-condition}) does not imply the general case, unlike for sequences of real numbers.
\end{ex}

The notion of strong $q$-log concavity has been proven for many sequences of combinatorial interest (see, e.g. \cite{Sagan-inductive,Sagan-symm-func}).  The following proposition illustrates the naturality of the definition of strong $q$-log concavity; to our knowledge it has not appeared before in the literature. The analogous statement for sequences of real numbers is well known.

\begin{prop} \label{prop:closed-under-convolution}
Let $(a_i(q))_{i=0,1,...,k}$ and $(b_i(q))_{i=0,1,...,\ell}$ be strongly $q$-log concave sequences.  Define the \emph{convolution} $(c_i(q))_{i=0,1,...,k+\ell}$ of these sequences by
\[
\left(\sum_i a_i(q) t^i \right) \left(\sum_i b_i(q) t^i \right)=\sum_i c_i(q) t^i.
\]
Then $(c_i(q))_{i=0,1,...,k+\ell}$ is a strongly $q$-log concave sequence.
\end{prop}
\begin{proof}
We will adapt a proof of Stanley \cite{Stanley-log-concave-survey} for the case of real numbers to the polynomial setting.

We make the convention that both sequences are zero outside of the given indexing sets. Define matrices $A=(a_{i+j}(q))_{i,j=0,1,...,k+\ell}$ and $B=(b_{i+j}(q))_{i,j=0,1,...,k+\ell}$, and notice that $AB=(c_{i+j}(q))_{i=0,1,...,k+\ell}$.  Condition (\ref{eq:strong-lc-condition}) implies that all $2 \times 2$ minors of $A$ and $B$ lie in $\R_{\geq 0}[q]$.  The Cauchy-Binet formula expresses the $2 \times 2$ minors of $AB$ as sums of products of such minors of $A$ and $B$, thus the $2 \times 2$ minors of $AB$ also lie in $\R_{\geq 0}[q]$; this implies condition (\ref{eq:strong-lc-condition}) for the sequence $(c_i(q))_i$.
\end{proof}

We remark that Proposition \ref{prop:closed-under-convolution} is false if one only imposes (\ref{eq:strong-lc-condition}) for $i=j$. Take $a_0=q^2$, $a_1=q+q^2$, $a_2=1+2q+q^2$, $a_3=4+2q+q^2$ so that $a$ satisfies $a_i^2-a_{i-1}a_{i+1}$ for all $i$ (this sequence is provided in \cite{Sagan-inductive}) and take $b_0=b_1=1$. We obtain $c_1=q+2q^2$, $c_2=1+3q+2q^2$ and $c_3=5+4q+2q^2$ with $c_2^2-c_1c_3=1+q-q^2+2q^3$.

\section{$q$-analogs of descent polynomials}
\label{sec:q-descent}

We write $\qbinom{n}{k}$ for the \emph{q-binomial coefficient}
\[
\qbinom{n}{k}=\frac{\qfac{n}}{\qfac{n-k}\qfac{k}}.
\]
There are many combinatorial interpretations of the $q$-binomial coefficient. We will mainly use the following:
\[
\qbinom{n}{k}=\sum_{w\in\mathfrak{S}_n,\Des(w)\subset\{k\}}q^{\ell(w)}=\sum_{w\in\mathfrak{S}_n,\Des(w)\subset\{n-k\}}q^{\ell(w)}.
\]
Similarly, we have the $q$-\textit{multinomial} coefficient
\[
\qbinom{n}{n_1,n_2,\ldots,n_k}=\frac{[n]!_q}{[n_1]!_q[n_2]!_q\cdots[n_k]!_q}=\qbinom{n}{n_1}\qbinom{n-n_1}{n_2}\cdots\qbinom{n_k}{n_k}
\]
for $n=n_1+\cdots+n_k$, with the combinatorial interpretation
\[
\qbinom{n}{n_1,n_2,\ldots,n_k}=\sum_{\substack{w\in\mathfrak{S}_n\\\Des(w)\subset\{n_1,\ldots,n_{k-1}\}}}q^{\ell(w)}.
\]

Theorem \ref{thm:q-des-formula} is a direct $q$-analog of Theorem \ref{thm:des-poly-a-basis}.

\begin{theorem} \label{thm:q-des-formula}
Let $S \subset \Z_{>0}$ be a finite set of positive integers and $m=\max(S)$, then:
\[
D_S(n,q)=\sum_{k=0}^m a_k(S;q) \qbinom{n-m}{k},
\]
where $a_0(S)=0$ and for $k \geq 1$ the polynomial $a_k(S;q) \in \mathbb{R}_{\geq 0}[q]$ is given by 
\[
a_k(S;q)=\sum_{\substack{\pi \in A(S;m+k) \\  [m+1,m+k]\subset \{\pi_1,...,\pi_m\}}} q^{\ell(\pi)}.
\]
\end{theorem}
\begin{proof}
The proof is modeled on the proof of Theorem \ref{thm:des-poly-a-basis} by Diaz-Lopez et al., taking account of the distribution of the lengths of the permutations involved. For $0\leq k\leq m$, let \[A^{(k)}(S;n)=\{\pi\in A(S;n):\{\pi_1,\ldots,\pi_m\}\cap[m+1,n]=[m+1,m+k]\}.\]
Notice that for $\pi\in A(S;n)$, $\pi(m+1)\leq m$ so $A^{(0)}(S;n)=\emptyset$. Also we see that for $\pi\in A^{(k)}(S;n)$, $\pi(j)=j$ for $j>m+k$. This directly implies that $\sum_{\pi\in A^{(k)}(S;n)}=a_k(S;q)$ by definition.

For $w\in A(S;n)$ such that $\{\pi_1,\ldots,\pi_m\}\cap[m+1,n]$ has cardinality $k$, we can construct a unique permutation $w'=f(w)\in A^{(k)}(S;n)$ such that $w'(i)=w(i)$ if $w(i)\leq m$ and the relative ordering of $w'(1),\ldots,w'(m)$ is the same as the relative ordering of $w(1),\ldots,w(m)$. In other words, to obtain $w'$ from $w$, we fix the entries with indices and values both not exceeding $m$, replace entries within $w(1),\ldots,w(m)$ that exceed $m$ by $m+1,\ldots,m+k$ in their original order, and require $w(m+1),\ldots,w(n)$ to be increasing. It is shown in \cite{descent-polynomials} that this map $f$ is well-defined, i.e., $f(w)\in A^{(k)}(S;n)$, and that it is ${n-m\choose k}$ to 1. Let $B_{\pi}(S;n)=f^{-1}(\pi)$. We see that 
\[
\ell(f(w))-\ell(w)=\#\{(i,j):i\leq m<j,w(i)>w(j)>m\},
\]
which fits well with the combinatorial interpretation for the $q$-binomial coefficients. Now
\begin{align*}
D_S(n,q)=&\sum_{k=1}^m\sum_{\pi\in A^{(k)}(S;n)}\sum_{w\in B_{\pi}(S;n)}q^{\ell(w)}\\
=&\sum_{k=1}^m\sum_{\pi\in A^{(k)}(S;n)}q^{\ell(\pi)}\qbinom{n-m}{k}\\
=&\sum_{k=1}^m a_k(S;q)\qbinom{n-m}{k}.
\end{align*}
\end{proof}

Theorem~\ref{thm:log-concavity-in-b-basis} below is a $q$-analogue of Corollary 4.2 in \cite{Bencs}.  Our proof is different from that given in \cite{Bencs} for the $q=1$ case, and in particular does not require Stanley's result \cite{Stanley-log-concave-survey} regarding log-concavity of sequences coming from linear extensions of posets, which relies on the difficult Aleksandrov-Fenchel inequalities in geometry. 

\begin{theorem} \label{thm:log-concavity-in-b-basis}
Let $S\subset\Z_{>0}$ be a finite set of positive integers and $m=\max(S)$, then:
\[
D_S(n,q)=\sum_{k=0}^m b_k(S;q)\qbinom{n-k}{m-k+1},
\]
where $b_0(S;q)=0$ and for $k\geq1$, the polynomial $b_k(S;q)\in\mathbb{R}_{\geq0}[q]$ is given by
\[
b_k(S;q)=\sum_{\substack{\pi\in A(S;m+1)\\\pi(m+1)=k}}
q^{\ell(\pi)}.\]
Moreover, $(b_k(S;q))_{k=0,\ldots,m}$ is a strongly $q$-log concave sequence.
\end{theorem}
\begin{proof}
We first justify the expansion regarding $D_S(n,q)$. As a piece of notation, for any permutation $w\in \mathfrak{S}_n$ and $k\leq n$, let $w|_k\in\mathfrak{S}_k$ be the permutation obtained from the relative ordering of $w(1),\ldots,w(k)$. Define $A_{\pi}(S;n):=\{w\in A(S;n): w|_{m+1}=\pi\}$. Fix $\pi\in\mathfrak{S}_{m+1}$ with $\pi(m+1)=k$ and consider $w\in A_{\pi}(S;n)$. We have 
\[
\ell(w)=\ell(\pi)+\#\{i<m+1<j:w(i)>w(j)\}. 
\]
Such $w$ is uniquely determined by $w(m+2),\ldots,w(n)$, which corresponds to a subset of $\{k+1,\ldots,n\}$ of size $n-m-1$. By the combinatorial interpretation of $q$-binomial coefficient provided in the beginning of this section, we see that
\[
\sum_{w\in A_{\pi}(S;n)}q^{\ell(w)}=q^{\ell(\pi)}\qbinom{n-k}{m-k+1}.
\]
As a result,
\begin{align*}
D_S(n,q)=&\sum_{\pi\in A(S;m+1)}\sum_{w\in A_{\pi}(S;n)}q^{\ell(w)}=\sum_{k=1}^m\sum_{\substack{\pi\in A(S;m+1)\\\pi(m+1)=k}}\sum_{w\in A_{\pi}(S;n)}q^{\ell(w)}\\
=&\sum_{k=1}^m\sum_{\substack{\pi\in A(S;m+1)\\\pi(m+1)=k}}q^{\ell(\pi)}\qbinom{n-k}{m-k+1}\\
=&\sum_{k=1}^mb_k(S;q)\qbinom{n-k}{m-k+1}.
\end{align*}

Next we show that $(b_k(S;q))_{k=1,\ldots,m}$ is strongly $q$-log concave. We make a slight generalization for the sake of induction. Define
\[
b_k(S,n;q):=\sum_{\substack{w\in A(S;n)\\w(n)=k}}q^{\ell(w)}
\]
so that $b_k(S;q)=b_k(S,m+1;q)$. We use induction on $n$ to show that $(b_k(S,n;q))_{k=1,\ldots,n}$ is strongly $q$-log concave. The base case $n=1$ with $S=\emptyset$ is trivial. There are two cases to be considered: $n-1\in S$ and $n-1\notin S$, which are analogous to each other. 

Assume that $n-1\notin S$, meaning that $w(n-1)<w(n)$ for $w\in A(S;n)$. For $w(n)=k$, we know $\ell(w)=\ell(w|_{n-1})+n-k$. Thus,
\[
b_k(S,n;q)=q^{n-k}\sum_{i=1}^{k-1} \sum_{\substack{w'\in A(S;n-1)\\w'(n-1)=i}}q^{\ell(w')}=q^{n-k}\sum_{i=1}^{k-1}b_i(S,n-1;q).
\]
By Proposition~\ref{prop:closed-under-convolution} with the sequence $a_i=1$ for all $i$, we know that if $(b_i)_{i\geq1}$ is $q$-log concave, then $(c_i)_{i\geq1}$ is $q$-log concave where $c_i=b_1+\cdots+b_i$. By induction hypothesis, $(b_i(S,n-1;q))_{i\geq1}$ is strongly $q$-log concave, so $(b_k(S,n;q))_{k\geq1}$ is strongly $q$-log concave as well.

Assume that $n-1\in S$. With the same argument, we obtain
$$b_k(S,n;q)=q^{n-k}\sum_{i=k}^{n-1}b_i(S\setminus\{n-1\},n-1;q)$$
which is also strongly $q$-log concave by induction hypothesis and convolution. 
\end{proof}

\begin{conj} \label{conj:log-concavity-in-a-basis}
The coefficients $(a_k(S;q))_{k\geq1}$ (defined in Theorem~\ref{thm:q-des-formula}) form a strongly $q$-log concave sequence.
\end{conj}

\begin{remark}
The coefficients $a_k(S;q)$ and $b_k(S;q)$ from Theorems \ref{thm:q-des-formula} and \ref{thm:log-concavity-in-b-basis} are related in the following way:
\[
a_k(S;q)=q^{k(k-1)}\sum_{i=1}^{m-k+1}\qbinom{m-i}{k-1}b_i(S;q).
\]
When $q=1$, the log-concavity of $(b_i)_{i\geq1}$ implies the log-concavity of $(a_k)_{k\geq1}$ for any sequences $b$ and $a$ related as above (see Theorem 2.5.4 of \cite{Brenti-thesis}). However, this is no longer true for strong $q$-log concavity. A counterexample is given by $m=4$, $b_1=1$, $b_2=q^{10}$, $b_3=q^{20}$ and $b_4=q^{30}$, where we compute:
\begin{align*}
    a_1=&q^{30}+q^{20}+q^{10}+1,\\
    a_2=&q^{22} + q^{13} + q^{12} + q^4 + q^3 + q^2,\\
    a_3=&q^{16} + q^{8} + q^7 + q^6,\\
    a_4=&q^{12},
\end{align*}
and $a_2^2-a_1a_3$ has negative coefficients.
\end{remark}

\section{$q$-analogs of peak polynomials}
\label{sec:q-peak}
We say that a polynomial $f\in\mathbb{R}[q]$ is \textit{palindromic in degree} $d$ if $f(q)=q^df(1/q)$; note that this does not necessarily mean that $f$ has degree $d$. It is easy to see that if both $f$ and $g$ are palindromic in degree $d$, so is $f+g$ and if $f$ is palindromic in degree $d_1$ and $g$ is palindromic in degree $d_2$, then $fg$ is palindromic in degree $d_1+d_2$. Moreover, it is well-known that the $q$-binomial coefficient $\qbinom{n}{k}$ is palindromic in degree $k(n-k)$. 

We say a set $S=\{s_1 < \cdots < s_r\}$ is $n$-\textit{admissible}, if $S=\Peak(\pi)$ for some $\pi \in \mf{S}_n$.  It is not hard to see that $S$ is $n$-admissible if and only if $s_1>1$ and $s_{i+1}>s_i+1$ for all $i$.  

We use the following standard notation, a special case of the $q$-Pochhammer symbol:
\[
(-q;q)_k = \prod_{i=1}^k (1+q^i).
\]

Given $S=\{s_1 < \cdots < s_r\}$ a nonempty $n$-admissible set, we say that $T \subset \Z_{>0}$ is \emph{$S$-compatible} if:
\begin{itemize}
    \item $T \cap S = \emptyset$, 
    \item $t < s_r$ for all $t \in T$, 
    \item at most one element of $T$ lies between any pair of consecutive elements of $S \sqcup \{0\}$ (in particular, $|T| \leq |S|)$, and
    \item if $s',s$ are consecutive elements of $S \sqcup \{0\}$ and $s-s'$ is odd, then there is an element $t \in T$ with $s' < t < s$.
\end{itemize}
By convention, the only $\emptyset$-compatible set is $\emptyset$.  When $T$ is understood, for $s \in S$ let $s^-$ be the largest element of $S \sqcup T \sqcup \{0\}$ less than $s$, and for $t \in T$ let $t^+$ be the smallest element of $S \sqcup T \sqcup \{0\}$ larger than $t$ (in fact $t^+$ will always lie in $S$, by the hypotheses).  For $T$ an $S$-compatible set, we define: 
\[
\e(S,T)=(-1)^{|S|+\sum_{t \in T}(t^+-t)}.
\]

\begin{theorem} \label{thm:q-peak-formula}
Let $S=\{s_1 < \cdots < s_r\}$ be $n$-admissible:
\[
P_S(n,q)=\sum_T \e(S,T) \prod_{i=1}^{r'+1} \qbinom{t_{i}}{t_{i-1}} (-q;q)_{t_{i}-t_{i-1}-1},
\]
where the sum is over all $S$-compatible sets $T=\{t_1 < \cdots < t_{r'}\}$, with the conventions that $t_0=0$ and $t_{r'+1}=n$.
\end{theorem}

\begin{remark} 
The term in Theorem \ref{thm:q-peak-formula} corresponding to the $S$-compatible set $T$ is divisible by $n-|T|-1$ factors of the form $(1+q^j)$.  Thus, since $|T| \leq |S|$, setting $q=1$ we recover the fact (\ref{eq:peak-is-power-times-poly}) that $P_S(n)$ is $2^{n-|S|-1}$ times an integer-valued polynomial $p_S(n)$.  The factors of the form $(1+q^j)$ differ from term to term, however, so it is not clear whether it is possible to produce a meaningful $q$-analog of $p_S(n)$ itself. 
\end{remark}

\begin{ex}
Let $S=\{2,5\}$, which is $n$-admissible for $n \geq 6$.  Note that $5-2=3$ is odd, so any $S$-compatible set $T$ must contain either $3$ or $4$.  The $S$-compatible sets are $T_1=\{1,3\}, T_2=\{1,4\}, T_3=\{3\}, T_4=\{4\}$.

We have $\e(S,T_2)=\e(S,T_3)=1$ and $\e(S,T_1)=\e(S,T_4)=-1$.  Thus for $n \geq 6$:
\begin{align*}
P_{\{2,5\}}(n,q)&=\qbinom{n}{4}\qbinom{4}{1}(-q;q)_{n-5}(-q;q)_{2}+\qbinom{n}{3}(-q;q)_{n-4}(-q;q)_{2} \\
& - \qbinom{n}{3}\qbinom{3}{1}(-q;q)_{n-4}(-q;q)_{1} - \qbinom{n}{4}(-q;q)_{n-5}(-q;q)_{3}.
\end{align*}
\end{ex}

The following corollary proves a conjecture stated by Alexander Diaz-Lopez in his talk at Discrete Math Days of the Northeast.  An alternative proof of Corollary \ref{cor:symmetry-conjecture} appears in Section \ref{sec:alternate-proof}.

\begin{cor} \label{cor:symmetry-conjecture}
For $n$ fixed and $S$ $n$-admissible, $P_S(n,q)$ is palindromic in degree $n\choose2$.
\end{cor}
\begin{proof}
It is well known that the $q$-binomial coefficients $\qbinom{x}{y}$ are palindromic in degree ${x \choose 2} - {y \choose 2} - {x-y \choose 2}$, and that products of palindromic polynomials are palindromic with degrees adding.  Thus it suffices to check that all terms in the sum in Theorem \ref{thm:q-peak-formula} are of the same degree in $q$, so that the sum preserves the symmetry.  Indeed, since $(-q;q)_k$ has degree ${k+1 \choose 2}$, the degree of each summand is a telescoping sum equal to ${n \choose 2}$.
\end{proof}

\subsection{Proof of Theorem \ref{thm:q-peak-formula}}

\begin{lemma} \label{lem:emptyset-peaks}
For $n\geq1$, $P_{\emptyset}(n,q)=(-q;q)_{n-1}=(1+q)(1+q^2)\cdots(1+q^{n-1})$.
\end{lemma}
\begin{proof}
Given $\pi'\in \mathfrak{S}_{n-1}$ with no peaks, we can insert $n$ at either the beginning or the end to obtain a permutation $\pi \in\mathfrak{S}_n$ with no peaks. In the first case $\ell(\pi)=\ell(\pi')$, and in the second $\ell(\pi)=\ell(\pi')+(n-1)$, so the lemma follows.
\end{proof}

Lemma \ref{lem:q-peak-recurrence} modifies an idea of Billey, Burdzdy, and Sagan \cite{Billey-Burdzy-Sagan} to take account of the lengths of the permutations involved.

\begin{lemma} \label{lem:q-peak-recurrence}
Suppose that $S=\{s_1<\cdots<s_r\}$ is $n$-admissible, then 
\[
P_S(n,q)=\qbinom{n}{k} \cdot P_{S_1}(k,q) \cdot (-q;q)_{n-k-1} - P_{S_1}(n,q) - P_{S_2}(n,q),
\]
where $S_1=S \setminus \{s_r\}$, $k=s_r-1$, and $S_2=S_1 \cup \{k\}$.
\end{lemma}
\begin{proof}
Consider the set $\Pi$ of permutations in $\mathfrak{S}_n$ such that $\Peak(\pi_1\cdots\pi_k)=S_1$ and $\Peak(\pi_{k+1}\cdots\pi_n)=\emptyset$. By choosing the first $k$ coordinates first, we have
\[
\sum_{\pi\in\Pi}q^{\ell(\pi)}=\qbinom{n}{k} \cdot P_{S_1}(k,q) \cdot P_{\emptyset}(n-k,q).
\]
On the other hand, since $\Peak(\pi)$ is one of $S,S_1,$ or $S_2$, and all possibilities are covered, we have
\[
\sum_{\pi\in\Pi}q^{\ell(\pi)}=P_{S}(n,q)+P_{S_1}(n,q)+P_{S_2}(n,q).
\]
Rearranging terms and applying Lemma \ref{lem:emptyset-peaks} completes the proof.
\end{proof}

We are now ready to complete the proof of Theorem \ref{thm:q-peak-formula}.

\begin{proof}[Proof of Theorem \ref{thm:q-peak-formula}]
Consider repeatedly applying the recursion in Lemma \ref{lem:q-peak-recurrence} in order to compute $P_S(n,q)$, continuing until the base case of Lemma \ref{lem:emptyset-peaks} is reached in every branch of the recursion.  The branches consist of a choice, each time the recursion is applied, of one of the terms $\qbinom{n}{k} P_{S_1}(k,q)(-q;q)_{n-k-1}$, or $-P_{S_1}(n,q)$, or $-P_{S_2}(n,q)$.  Thus we think of the branches as sequences of sets, beginning with $S$, modified one step at a time, and ending with $\emptyset$; branches also accrue weights corresponding to the coefficients of the $P$'s in the above three terms.  The information about a branch which is relevant to computing its weight will be encoded in a set $\tilde{T}$.  For branches with nonzero weight we will see that $\tilde{T}$ is of the form $T \sqcup \{n\}$, with $T$ an $S$-compatible set indexing a summand in Theorem \ref{thm:q-peak-formula}.  

Start with fixed $S$ and $n$ as desired in order to compute $P_S(n,q)$; let $\tilde{T}=\{n\}$.  The possible modifications are:
\begin{enumerate}
    \item[(i)] Replace $S=\{s_1 < ... < s_r\}$ with $S_1=\{s_1,...,s_{r-1}\}$ (that is, delete $s_r$) and multiply the weight of the branch by $\binom{n}{s_r-1}(-q;q)_{n-s_r}$.  In this case we add $s_r-1$ to $\tilde{T}$.  Replace $n$ with $s_r-1$.
    \item[(ii)] Replace $S=\{s_1 < ... < s_r\}$ with $S_1=\{s_1,...,s_{r-1}\}$ (that is, delete $s_r$) and multiply the weight of the branch by $x$ ($x$ will later be specialized to $-1$ as in the second term in the recurrence) and leave $\tilde{T}$ and $n$ unchanged.
    \item[(iii)] Replace $S=\{s_1 < ... < s_r\}$ with $S_2=\{s_1,...,s_{r-1},s_r-1 \}$ (that is, decrement $s_r$ by one) and multiply the weight of the branch by $x$.  Again leave $\tilde{T}$ and $n$ unchanged.  If $S_2$ is no longer admissible, meaning $s_r-1=s_{r-1}+1$, then the weight of this branch becomes 0, since $P_{S_2}(n,q)=0$.
\end{enumerate}

We now observe several facts:
\begin{itemize}
    \item[(a)] The sets $\tilde{T}$ which correspond to branches with nonzero weights are exactly those of the form $T \sqcup \{n\}$ with $T$ an $S$-compatible set.
    \item[(b)] After specializing $x=-1$, the weight of a branch depends, up to sign, only on the corresponding set $\tilde{T}$, and this weight is the product appearing in Theorem \ref{thm:q-peak-formula}.
    \item[(c)] After grouping branches according to $\tilde{T}$, the polynomial $f_{\tilde{T}}(x)$ coming from summing the powers of $x$ associated to each branch satisfies $f_{\tilde{T}}(-1)=\e(S,\tilde{T} \setminus \{n\})$.
\end{itemize}
Together, these facts imply Theorem \ref{thm:q-peak-formula}, as $P_S(n,q)$ is the weighted sum over all branches and (a), (b), and (c) imply that this weighted sum is given by the formula in the theorem.

To see that (a) holds, note that we only add an element to $\tilde{T}$ in case (i). Let $T=\tilde{T} \setminus \{n\}$, and suppose that $\tilde{T}$ corresponds to a branch with a nonzero weight.  The largest possible element added to $\tilde{T}$ is $s_r-1$, so $\max(T)<s_r$.  Because $S$ is admissible, the added element $s_r-1 \not \in S$, so $T \cap S = \emptyset$.  Since only the largest element of $S$ is modified in each step, and since we only add to $\tilde{T}$ when we delete an element of $S$, it is not possible for two elements of $\tilde{T}$ to lie between consecutive elements of $S$. For the last condition in the definition of $S$-compatibility, see the argument for part (c) below.  Conversely, any $S$-compatible set $T$ can clearly be realized, since we may use case (iii) to decrement the largest element of $S$ until (i) can be applied to add a desired element to $\tilde{T}$; if $s_{r-1}$ is larger than the element we wish to add to $\tilde{T}$, simply delete $s_r$ using case (i) or (ii) and continue as before.

Except for the factors of $x$ which are produced in steps (ii) and (iii), the only time the weight of the branch is modified is in step (i), and $\tilde{T}$ is modified at the same time.  By design, the product $\prod \qbinom{t_i}{t_{i-1}}(-q;q)_{t_i-t_{i-1}-1}$ associated to $T$ changes in exactly the same way as the weight of the branch.  The $x$-weight of the branch is always $0$ or $\pm 1$ as discussed below.  This proves (b).

Given $s \in S$ and a set $T$ which satisfies the first three conditions of $S$-compatibility, if $s^- \in S \sqcup \{0\}$, then any branch yielding $T$ must not use case (i) when $s$ is the maximal element.  Thus it must be that case (iii) is applied some number of times, moving the maximal element left, and then (ii) is applied to delete it, without modifying $\tilde{T}$.  This contributes a factor of $x(1+x+\cdots+x^{s-s^--2})$ to the $x$-weight of $T$.  In particular, if $s-s^-$ is odd, when we specialize $x=-1$ we get 0, thus, in order for $T$ to contribute a term in the sum, $T$ must indeed be $S$-compatible; if $s-s^{-}$ is even, we get a factor of $-1$.  If instead we have $s^- \in T$, case (iii) must be used to move the maximal element to $s^-+1$ and then case (i) applied.  After setting $x=-1$ this results in a weight of $(-1)^{s-s^--1}$.  Now, assuming $T$ is $S$-compatible, its weight is nonzero.  We have gained a factor of $-1$ for each $s \in S$ with $s^- \not \in T$, for a total of $|S|-|T|$ factors.  We have also gained a factor of $(-1)^{s-s^--1}=(-1)^{t^+-t-1}$ for each $s^-=t \in T$.  In total we have sign
\[
(-1)^{|S|+|T|+\sum_{t \in T}(t^+-t-1)} = (-1)^{|S|+\sum_{t \in T}(t^+-t)}.
\]
This proves (c), completing the proof.
\end{proof}

\subsection{Another proof of Corollary~\ref{cor:symmetry-conjecture}} \label{sec:alternate-proof}
In this section, we present a separate proof of Corollary~\ref{cor:symmetry-conjecture} by writing $P_S(n,q)$ as an alternating sum of products of $q$-multinomial coefficients and $q$-Eulerian polynomials. Define
\[
Q_S(n,q):=\sum_{\substack{\pi\in\mathfrak{S}_n\\\Peak(\pi)\supset S}}q^{\ell(\pi)}=\sum_{S'\supset S}P_{S'}(n,q).
\]
The principle of inclusion and exclusion immediately gives us the following lemma.
\begin{lemma}\label{lem:PIE}
For $n$-admissible set $S$, \[
P_S(n,q)=\sum_{n\text{-admissible }S'\supset S}(-1)^{|S'|-|S|}Q_{S'}(n,q).\]
\end{lemma}
We now compute $Q_S(n,q)$ explicitly. For $S=\{s_1<\cdots<s_r\}$, we divide $S$ into blocks $S=S_1\sqcup S_2\sqcup\cdots \sqcup S_k$ such that $\max S_i<\min S_{i+1}$, and $s_j$ and $s_{j+1}$ belong to the same block if and only if $s_{j+1}-s_j=2$. For $S_i=\{s_j<\cdots<s_{j+a-1}\}$, let $\bar S_i:=\{s_j-1,\ldots,s_{j+a-1}+1\}$ of cardinality $2m+1$. By definition, $\bar S_i$ does not intersect $\bar S_j$ for distinct $i\neq j$. So we can partition $\{1,2,\ldots,n\}$ into consecutive intervals in the order of $U_0\sqcup \bar S_0\sqcup U_1\sqcup \bar S_1\sqcup\cdots\sqcup \bar S_k\sqcup U_k$.

As an example, take $n=12$, $S=\{3,5,8\}$. Then $S_1=\{3,5\}$ and $S_2=\{8\}$. Moreover, $\bar S_1=\{2,3,4,5,6\}$ and $\bar S_2=\{7,8,9\}$. Further, $U_0=\{1\}$, $U_1=\emptyset$ and $U_2=\{10,11,12\}$.

\begin{prop}\label{prop:expression-Q}
For an $n$-admissible set $S$ and $\{1,2,\ldots,n\}=U_0\sqcup \bar S_0\sqcup U_1\sqcup \bar S_1\sqcup\cdots\sqcup \bar S_k\sqcup U_k$ as above. Let $|U_i|=u_i$ and $|\bar S_i|=r_i$. Then
\[
Q_S(n,q)=\qbinom{n}{u_0,r_1,u_1,\ldots,r_k,u_k}\cdot\prod_{j=1}^k E_{r_j}(q)\cdot \prod_{i=0}^k[u_i]!_q
\]
where \[E_m(q)=\sum_{w\in\mathfrak{S}_m,w(1)<w(2)>w(3)<\cdots}q^{\ell(w)}\] is the length generating function for alternating permutations of size $m$.
\end{prop}
\begin{proof}
Requiring $w\in\mathfrak{S}_n$ to satisfy $\Peak(w)\supset S$ is the same as requiring that $w$ is an alternating permutation when restricted to $\bar S_i$ for each $i=1,\ldots,k$. Thus, the formula follows because the $q$-multinomial coefficient corresponds to assigning values of $\{1,2,\ldots,n\}$ to blocks $U_0,\bar S_1,U_1,\ldots,\bar S_k,U_k$ while keeping track of the number of inversions between distinct blocks, $E_{r_j}(q)$ is the length generating function for the block $\bar S_j$ and $[u_i]!_q$ is the length generating function for the block $U_i$ where no conditions are imposed.
\end{proof}

Now Corollary~\ref{cor:symmetry-conjecture} follows.
\begin{proof}[Proof of Corollary~\ref{cor:symmetry-conjecture}]
We first show that $Q_S(n,q)$ is palindromic in degree $n\choose2$ by Proposition~\ref{prop:expression-Q}. The $q$-multinomial coefficient on the right hand side of Proposition~\ref{prop:expression-Q} is known to be palindromic in degree ${n\choose2}-{u_0\choose2}-{r_1\choose2}-\cdots-{u_k\choose2}$, by a straightforward calculation. Each $r_j$ is odd by construction so for an alternating permutation $w\in\mathfrak{S}_{r_j}$, its reverse $w'$ is also alternating and $\ell(w)+\ell(w')={r_j\choose2}$. This means $E_{r_j}(q)$ is palindromic in degree ${r_j\choose2}$. In addition, $[u_i]!_q$ is palindromic in degree ${u_i\choose 2}$. Multiplying terms together, we see that $Q_S(n,q)$ is palindromic in degree $n\choose2$. As an alternating sum of palindromic polynomials in degree $n\choose2$ (Lemma~\ref{lem:PIE}), we conclude that $P_S(n,q)$ is palindromic in degree $n\choose2$.
\end{proof}

\section*{Acknowledgements}

We are grateful to Alexander Diaz-Lopez for piquing our interest in descent and peak polynomials, to Alex Postnikov for his helpful comments.  We also wish to thank Bruce Sagan for his careful reading of an earlier version of the paper and the Northeast Combinatorics Network for organizing the Discrete Math Days conference at which this project was initiated.

\bibliographystyle{plain}
\bibliography{q-peaks}
\end{document}